\documentclass[a4paper,11pt,oneside,reqno]{article}
\usepackage{amsbsy,amsfonts,amssymb,amsmath}
\usepackage{amsthm, mathrsfs,bbm}

\usepackage{lmodern}
\usepackage[T1]{fontenc}

\makeatletter
\renewcommand{\@seccntformat}[1]{\csname the#1\endcsname.\,\,}
\renewcommand\section{\@startsection {section}{1}{\z@}{-2ex \@plus -1ex \@minus -.2ex}{2.3ex \@plus.2ex}{\centering\large\bfseries}}

\renewcommand\subsection{\@startsection {subsection}{1}{\z@}{-3ex \@plus -1ex \@minus -.2ex}{2.3ex \@plus.2ex}{\centering\normalfont\bfseries}}

\newcommand{\adddot@section}{.}
\makeatother
\newtheorem{theorem}{\noindent Theorem}
\newtheorem{lemma}{\noindent Lemma}

\newtheorem{proposition}{\noindent Proposition}

\newtheorem{remark}{\noindent Remark}

\newcounter{ex}

\newcommand{\ind}[1]{\mathbbm{1}_{#1}}

\newcommand{\la}{\lambda}
\newcommand{\sR}{\mathbb{R}}

\newcommand{\set}[1]{\{#1\}}

\newcommand{\mset}[2]{\{#1:#2\}}

\newcommand{\kla}[1]{(#1)}

\newcommand{\bkla}[1]{\left(#1\right)}

\newcommand{\esym}{{\mathbbm E}}
\newcommand{\psym}{{\mathbbm P}}

\newcommand{\EE}[1]{\esym[#1]}
\newcommand{\PP}[1]{\psym(#1)}

\newcommand{\EX}[2]{\esym_{#1}[#2]}

\newcommand{\micro}[1]{{\scriptscriptstyle #1}}

\newcommand{\mi}{\wedge}
\newcommand{\ma}{\vee}

\newcommand{\maxs}[3]{\vee_{#1=#2}^{#3}}
\newcommand{\Mins}[3]{\bigwedge_{#1=#2}^{#3}}

\newcommand{\limY}{Z}
\newcommand{\arro}{\Rightarrow}
\newcommand{\dto}{\stackrel{\micro{d}}{\arro}}

\newcommand{\disas}{\stackrel{\micro{d}}{=}}

\newcommand{\lex}[1]{{\itshape #1}}

\newcommand{\zz}{\eta}
\newcommand{\ku}[1]{^{\scriptscriptstyle(#1)}}
\newcommand{\qq}[1]{\eta\ku #1}

\newcommand{\ee}{\widehat\eta}
\newcommand{\zx}{\zz_0}
\newcommand{\tra}[1]{{\mathscr P}_{#1}}
\newcommand{\trb}[2]{{\mathscr P}_{#2}\ku#1}

\newcommand{\gen}{{\mathscr A}}
\newcommand{\len}{{\mathscr L}}
\newcommand{\ben}[1]{{\mathscr A\ku{#1}}}

\newcommand{\dom}[1]{{\mathcal D}_{#1}}
\newcommand{\cnt}{{\mathcal C}_{0}}

\newcommand{\cond}[1]{{\rm (S#1)}}

\renewcommand{\phi}{\varphi}

\newcommand{\DD}{{\mathcal D}[0,\infty)}

\begin{document}
\abovedisplayskip=7pt plus 1pt minus 4pt
\belowdisplayskip=7pt plus 1pt minus 4pt
\belowdisplayshortskip=3pt plus 1pt minus 2pt

\title{Weak Convergence of Subordinators to Extremal Processes}

\author{
Offer Kella\thanks{Department of Statistics; The Hebrew
University of Jerusalem; Mount Scopus, Jerusalem 91905; Israel
({\tt Offer.Kella@huji.ac.il}).}\ \thanks{Supported in part by grant No. 434/09 from the Israel Science Foundation and the Vigevani Chair in Statistics.} \and Andreas L\"{o}pker\thanks{Department of
Economics and Social Sciences; Helmut Schmidt
University Hamburg, 22043 Hamburg;
Germany ({\tt lopker@hsu-hh.de}).}}

\maketitle
\begin{abstract} For certain subordinators $(X_t)_{t\geq 0}$ it is shown that the  process $(-t\log X_{ts})_{s>0}$ tends to an extremal process $(\ee_s)_{s>0}$ in the sense of convergence of the finite dimensional distributions.  Additionally it is also shown that $(z\mi (-t\log X_{ts}))_{s\ge 0}$ converges weakly to $(z\mi\ee_s)_{s\ge0}$ in $\DD$, the space of c\`adl\`ag functions equipped with Skorohod's $\mbox{J}_1$ metric.
\end{abstract}

\section{Introduction}
It was shown in \cite{LevEnis} that if $(X_t)_{t>0}$ is a family of positive random variables and if $X$ is a
non-constant random variable with distribution function $F$, then $X_{t}^{-t}$ converges weakly to $X$ as $t\to 0$ if and only if $\psi_{t}(u^{1/t})\to 1-F(u)$ as $t\to 0$ at all
continuity points $u$ of $F$, where $\psi_t$ is the Laplace transform of $X_t$.  In \cite{AL-BAR} it was found that for the convolution family  $\psi_t(u)=\varphi(u)^t$, where $\varphi$ is the Laplace transform of an infinitely divisible random variable, i.e. if the process $X_t$ is a subordinator, the limit distribution, if not concentrated on a single point, is always a Pareto distribution. Equivalently we can formulate the convergence in terms of the convergence of $-t\log X_t$ as $t$ tends to zero, with the only possible limit distribution being the exponential distribution. We will apply and extend these results to show that in fact the process $(-t\log X_{st})_{s>0}$ converges to a, so called, {\it extremal process} $(\ee_s)_{s>0}$, to be reviewed in Section~3. We will first observe the convergence of the finite dimensional distributions and then establish weak convergence of a truncated version in $\DD$, the space of c\`adl\`ag functions equipped with Skorohod's $\mbox{J}_1$ metric. Since the prelimit and limit processes are Markovian, this will be done by proving uniform convergence of the associated generators and applying the necessary theory from \cite{EK} for this setup.

\section{Setup, review and convergence of finite dimensional distributions}
Let $(X_{t})_{t\geq 0}$ be a pure jump subordinator, i.e. an increasing L\'evy process with
\begin{equation}\psi _{t}(u)={\mathbbm E}(e^{-uX_{t}})=e^{-t\phi (u)}\ ,\end{equation}
where
\begin{equation}
\phi (u)=\int_{0}^{\infty }(1-e^{-ux})\,d\nu (x)\ .
\end{equation}
and the L\'evy measure $\nu$ in this case must satisfy $\nu(-\infty,0]=0$, $\nu(1,\infty)<\infty$ and
\begin{equation}
\rho=\int_{[0,1]}u\,d\nu (u)=\int_0^1\nu(x,1]dx<\infty\ .
\end{equation}
 We recall that  $G_t(x)=\PP{X_t\leq x}$ is an infinitely divisible distribution.

In what follows $\mi$ and $\ma$ denote minima and maxima (respectively),  $\disas$ denotes equality in distribution, $\dto$ is for convergence in distribution and $x\downarrow x_0$ means  $x\to x_0,\ x>x_0$. Finally, for finite $\gamma>0$, denote by $E_\gamma$ an exponential random variable with mean $1/\gamma$.

In \cite{AL-BAR} the following result was proved.
\begin{theorem}
\label{t1} Let $\limY$ be a positive random variable
which is not concentrated at one point and let $F(x)=\PP{Z\leq x}$. The following statements are
equivalent:
\begin{enumerate}
\item[{\cond 1}] \label{a1} $-t\log X_t\dto
\limY$ as $t\downarrow 0$.
\item[{\cond 2}] \label{a2} $t\phi(u^{1/t})\to - \log(1-F(u))$ as $t\downarrow 0$, for all continuity points $u$ of $F$.
\item[{\cond 3}] \label{a3} $-t\log X_t\dto E_\gamma$ as $t\downarrow 0$ for some finite $\gamma >0$.
\end{enumerate}
Furthermore, for any finite $\gamma >0$ the following statements are equivalent:
\begin{enumerate}
\item[{\cond 4}] \label{a4} $-t\log X_t\dto E_{\gamma}$ as $t\downarrow 0$.

\item[{\cond 5}] \label{a5} $\phi(s)/\log s\to \gamma$ as $s\to\infty$.

\item[{\cond 6}] \label{a6} $\log G_1(x)/\log x\to \gamma$ as $x\downarrow 0$.

\item[{\cond 7}] \label{a7} $\nu(x,\infty)/\log x\to -\gamma$ as $x\downarrow 0$.
\end{enumerate}
\end{theorem}

Note that since $\nu(\epsilon,\infty)<\infty$ for any $\epsilon>0$, then $\cond{7}$ is equivalent to
\begin{enumerate}
\item[{\cond {7'}}]$ \label{a7'} \nu(x,\epsilon]/\log x\to -\gamma$ as $x\downarrow 0$.
\end{enumerate}
Also note that this condition cannot hold for a compound Poisson process, so that when it does hold then necessarily $\nu(0,\epsilon]=\infty$, which in turn implies that $X_t>0$ almost surely for each $t>0$ and thus $-t\log X_t$ is well defined for all $t>0$.

Several examples of subordinators fulfilling these conditions are given in \cite{AL-BAR}. A prominent member is the gamma process, where
\begin{align}
G_t(x)=\frac{\la^{\gamma}}{\Gamma(\gamma)}\int_0^x u^{\gamma-1}e^{-\la u}\,du.
\end{align}

The following is a generalization of Proposition 2.2 of \cite{AL-BAR} to the multidimensional and dependent case.
\begin{proposition}\label{prop:1}
For each $t>0$, let $(X_{i,t})_{1\le i\le n}$ be a random vector with almost surely positive components and assume that for some random vector $(X_i)_{1\le i\le n}$,
\begin{equation}
(-t\log X_{i,t})_{1\le i\le n}\dto (X_i)_{1\le i\le n}
\end{equation}
Then,
\begin{eqnarray}
\left(-t\log \left(\sum_{i=1}^k X_{i,t}\right)\right)_{1\le k\le n}\dto \left(\bigwedge_{i=1}^kX_i\right)_{1\le k\le n},\label{jj}
\end{eqnarray}
as $t\downarrow 0$.
\end{proposition}

\begin{proof}
It is well known that on a possibly different probability space we can take
$(\tilde X_{i,t})_{1\le i\le n}\disas (X_{i,t})_{1\le i\le n}$ and $(\tilde X_i)_{1\le i\le n}\disas(X_i)_{1\le i\le n}$,
where
\begin{equation}
(-t\log \tilde X_{i,t})_{1\le i\le n}\dto (\tilde X_1)_{1\le i\le n}
\end{equation}
almost surely.
Since any (Borel) function of $(\tilde X_{i,t})_{1\le i\le n}$ is distributed like that of $(X_{i,t})_{1\le i\le n}$ (and similarly for the limits) this implies that it suffices to show the validity of this proposition for the deterministic case, where the multidimensional convergence in (\ref{jj}) is equivalent to the convergence of each coordinate separately. Observing each such coordinate, it is apparent that it suffices to show this for the case $n=2$ and then proceed by induction. This can be concluded from Proposition~2.2 of \cite{AL-BAR}, but we would also like to point out the straightforward alternative below.

Note that if $-t\log a(t)\to a$ and $-t\log b(t)\to b$ then $-t\log(a(t)\wedge b(t))=(-t\log a(t))\vee(-t\log b(t))\to a\vee b$ and, similarly, $-t\log(a(t)\vee b(t))\to a\wedge b$, all as $t\downarrow 0$. Since $a(t)\wedge b(t)+a(t)\vee b(t)=a(t)+b(t)$ it therefore follows that it suffices to treat the case where $a(t)\ge b(t)$ for all $t>0$ and $a\le b$. For this case we have that
\begin{equation}
0\le \log(a(t)+b(t))-\log a(t)=\log\left(1+\frac{b(t)}{a(t)}\right)\le\log 2
\end{equation}
and thus $t\log(a(t)+b(t))-t\log a(t)\to 0$ as $t\downarrow 0$ and the proof is complete.
\end{proof}

\begin{remark}\label{re:1}\rm
Of course, if we assume in Proposition~\ref{prop:1} that $(X_{i,t})_{1\le i\le n}$ are independent, then $(-t\log X_{i,t})_{1\le i\le n}\dto (X_i)_{1\le i\le n}$ if and only if  $-t\log X_{i,t}\dto X_i$ for each $i$ and $(X_i)_{1\le i\le n}$ are independent as well (on an appropriate probability space). This will be needed in what follows.
\end{remark}

We now recall that if, in Proposition~\ref{prop:1}, $(X_{i,t})_{t\ge0}$ are independent subordinators, then $X_i$ are independent and are either constant or necessarily exponential. Thus, when they are all exponential, the distribution of the $k$th coordinate on the right side of \eqref{jj} is exponential as well, with parameter given by the sum of the first $k$ parameters for the the individual limits.

Now let $0=s_0<s_1<s_2<\ldots<s_n$ and, for $i=1,\ldots,n$, let $(X_{i,t})_{t\ge 0}$ be i.i.d. copies of $(X_t)_{t\ge 0}$. It follows from the stationary and independent increment property of the L\'evy process $X_t$ that
\begin{align}
X_{s_kt}=\sum_{i=1}^k \kla{X_{s_it}-X_{s_{i-1}t}}
\disas \sum_{i=1}^k X_{i,(s_{i}-s_{i-1})t}.
\end{align}
Consequently, with $Z_1,Z_2,\ldots$ being i.i.d. $\exp(1)$ random variables (so that $Z_i/\beta\disas\exp(\beta)$), applying \eqref{jj}  it follows that, as $t\downarrow0$,
\begin{align}\label{eq:fidi}
\left(-t\log X_{s_1t},-t\log X_{s_2t},\ldots,-t\log X_{s_nt}\right)\dto\frac{1}{\gamma}\left(\bigwedge_{i=1}^k\frac{Z_i}{s_i-s_{i-1}}\right)_{1\le k\le n}\ .
\end{align}

Hence, we see that we have convergence of the finite dimensional distributions of
$(-t\log X_{ts})_{s>0}$ to those of
some process $(\ee_s)_{s>0}$, where $\bkla{\ee_{s_1},\ldots,\ee_{s_n}}$ is distributed like the right hand side of (\ref{eq:fidi}).

In the next section we will identify this process, which turns out to be a known one and then show in the following section that the convergence of a truncated version of the process above holds in the sense of weak convergence in $\DD$.

\section{The extremal process}\label{sec:extremal}
 Recall that $Z_1,Z_2,\ldots$ are i.i.d. $\exp(1)$ random variables and let $M_n=\frac{1}{\gamma}\Mins k1n Z_k$. Then the process $n\cdot M_{[tn]+1}$ converges as $n\to\infty$ weakly to a process $\ee_t$, the so called  \lex{extremal process} (\cite{dwass}). This process has the following properties (see Section 4.3 in \cite{Resnick}):
\begin{enumerate}\setlength{\itemsep}{0ex}
\item $\ee$ is stochastically continuous and has a version in $\DD$ (from hereon this is the assumed version).
\item\label{huh} $\ee$ has non-increasing paths, is piecewise constant,  almost surely $\lim_{s\to 0}\ee_s=\infty$ and
$\lim_{s\to\infty}\ee_s=0$.
\item\label{finix} the finite dimensional distributions are given by the right hand side of (\ref{eq:fidi}),
in particular
\begin{align}\PP{\ee_{s_i}> x_i,i=1,\ldots,n}
= \exp\left(-\gamma\sum_{i=1}^n \maxs jik x_j\right)\ .
\end{align}
\item The holding times in $x$ are exponential with rate  $\gamma x$.
\item If the process jumps at time $t$  then  $\ee_t=\ee_{t-}\cdot U$, where $U$  is independent of $\set{X_s,0\leq s<t}$ (in an appropriate sense) and has a uniform distribution in $[0,1]$.
\end{enumerate}

Now let $\zx$ be a random variable, independent of $\{\ee_t,t\geq 0\}$ and define \begin{equation}\label{eq:eta}\zz_t:=\ee_t\mi\zx\ .\end{equation} The processes $\zz_t$ is a Markov process that inherits the above properties 1-5 from $\ee$, except for
\begin{itemize}
\item[\ref{huh}$^\ast$.] $\zz$ has non-increasing paths, is piecewise constant,  almost surely
$\lim_{s\to\infty}\zz_s=0$.
\item[\ref{finix}$^\ast$.] The finite dimensional distributions are given by
\begin{align}
\bkla{\zz_{s_1},\zz_{s_2},\ldots,\zz_{s_n}}\disas\left(\eta_0\wedge\frac{1}{\gamma}\left(\bigwedge_{i=1}^k\frac{Z_i}{s_i-s_{i-1}}\right)\right)_{1\le k\le n} \ .\label{5o}
\end{align}
\end{itemize}
For a proof note that the first jump below $\zx$ of the process $\zz$ will go uniformly into the interval $[0,\zx]$.  Since from then on the process $\zz$ will continue just like $\ee$, we only have to show that the holding time in $\zx$, given by $T=\inf\mset{t>0}{\zz_t\leq \zx}$,  has an exponential distribution with rate $\gamma \zx$. Indeed, we have for all $s>0$,
$\PP{T>s|\zx}=\PP{\ee_s>\zx|\zx}=e^{-\gamma s\zx}$. The property 3$^\ast$ is obvious from the construction.

It follows from the above properties that the transition probabilities of the Markov process $\zz$ are given by
\begin{align}\label{li}
\PP{\zz_{s+t}> x|\zz_s=y}=\ind{\{x< y\}}\exp\kla{-\gamma t x},\quad t,s\geq 0.
\end{align}
Hence, for bounded functions $f:\sR\to\sR$ the transition semi-group of the process is given by
\begin{align}
\tra tf(x):=\EX x{f(\zz_t)}=e^{-\gamma tx}f(x)+\gamma t \int_0^x f(y) e^{-\gamma ty}\,dy
\end{align}
and hence the limit
\begin{align}
\lim_{t\to 0}\frac{\EX x{f(\zz_t)}-f(x)}{t}\
&=-\gamma  x f(x)+\gamma \int_0^x f(y)\,dy
\end{align}
exists uniformly at least for $f\in\cnt$, where $\cnt$ is the class of continuous functions
 $f:\sR\to\sR$ that vanish as $|x|\to\infty$. Moreover, the Feller property holds, i.e. $\tra t\cnt\subset\cnt$ and $\tra tf(x)\to f(x)$ as $t\to 0$ for $f\in \cnt$.

 For $f\in\cnt$ the generator of the Markov process $\zz$ is then given by
\begin{align}
\gen f(x)
&=\gamma x\int_0^1  (f(xy)-f(x))\, dy=\gamma\int_0^x  (f(y)-f(x))\, dy.
\end{align}
We choose a smaller domain, namely those functions $f\in\cnt$ which are differentiable with derivative $f'\in\cnt$
(let $\dom\gen$ denote this class).
Then we can write
\begin{align}
\gen f(x)=-\gamma\int_0^x  u f'(u)\, du.
\end{align}
We enlarge the state space from $(0,\infty)$ to $\sR$ by setting $\gen f(x)=0$ for $x\le 0$. The reason is, that the process $-t\log X_{ts}$ will have values in $\sR$ rather than in $(0,\infty)$. Hence, by construction, $\zz_t$ will stay constant, if started in $x\leq 0$. Note that if $f\in\dom A$ then also $\tra tf\in\dom A$ since for $x\geq 0$
\begin{align}
(\tra tf)'(x)=e^{-\gamma tx}(f'(x)-\gamma tf(x))+\gamma t e^{-\gamma tx} f(x) =e^{-\gamma tx}f'(x).\label{qqq}
\end{align}

\section{Convergence in $\boldsymbol{\DD}$}
Recalling~(\ref{eq:eta}), the following is the main result of this paper.
\begin{theorem}
Suppose that the subordinator $(X)_{t\ge 0}$ satisfies one of the conditions of Theorem \ref{t1} and that $z\in(0,\infty)$. Then
\begin{align}
(z\wedge(-t\log X_{ts}))_{s\ge0}\dto (\zz_s)_{s\ge0}
\end{align}
as $t\to 0$ weakly in $(\DD,\mbox{J}_1)$ and  $\zx=z$.
\end{theorem}

\begin{proof} Let us write $X_t=X_t'+X_t''$, where $X_t'$ has L\'evy measure $\nu'(A)=\nu(A\cap (0,1])$ and $X_t''$ has L\'evy measure $\nu''(A)=\nu(A\cap[1,\infty))$. That is, $X_t'$ captures the small jumps and $X_t''$ is a compound Poisson process with jumps at least of size one. It is well known that $X_t'$ and $X_t''$ are independent. Moreover, $X''_t=0$ for $t<\kappa$, where $\kappa$ is an exponential random variable so that, assuming $t$ to be small enough $X_t=X_t'$. Since we are interested in the limiting behaviour as $t\to 0$, we may assume that $\nu$ is concentrated on $(0,1)$. Then $X_t$ is a Markov process with generator (\cite{bertoin}) given by
\begin{align}
\len f(x)=\int_0^1 (f(x+y)-f(x))\nu(dy),\quad x\geq 0
\end{align}
for functions $f\in\dom\gen$. For fixed $t$ the process $\qq t_s=-t\log X_{ts}\mi z$ is a Markov process with sample paths in $\DD$. The time-change $X_s\to X_{ts}$ transforms $\len f$ into $t\len f(x)$, while the subsequent state-space transformation $X_t\to g(X_t)$, with $g(x)=-\log x$, changes $t\len f(x)$ to $t(\len f\circ g)(g^{-1}(x))$, see e.g. \cite{dyn}.  Hence the generator of the process $\qq t_s$ is given by
\begin{align}
\ben t f(x)&=t\int_0^1 (f(-t\log(y+e^{-x/t}))-f(x))\,\nu(dy).
\end{align}
For the transition semi-group of $\qq t$ we obtain
\begin{align}
\trb ts f(x)=\EE{f(-t\log X_{ts} \mi x)}
\end{align}
Hence $\trb tsf$ is continuous for $f\in\cnt$. Moreover $|\trb tsf(x)|\to 0$ as $|x|\to\infty$ by dominated convergence and $\trb tsf(x)\to f(x)$ as $s\to 0$ by dominated convergence and the fact that $-t\log X_{ts}\to\infty$ as $s\to 0$. Hence for every $t>0$ the process $\qq t$ has the Feller-property.

 In Lemma~\ref{uniform} to follow we will show that, for every $z>0$, $\ben t f\to \gen f$ uniformly on $(-\infty,z]$. As the process is nonincreasing and thus, one does not need to consider uniform convergence on the entire state space $\mathbb{R}$, it will follow from Theorem 6.1, p.28 in \cite{EK} that the respective transition operators converge, too, provided that $\dom A$ is a core for the generator. But this follows from Proposition 3.3, p.17 in \cite{EK} since $\dom A$ is dense in $\cnt$ and $\tra tf\in\dom A$ if $f\in\dom A$ (as was shown in \eqref{qqq}). From Theorem 2.5, p.167 in \cite{EK} it then follows, using the Feller-property of $\qq t$, that $\qq t$ tends to $\zz$ in $\DD$.  Since $-t\log X_{ts}$ tends to $\infty$ as $s\to 0$, it is clear that $\zx=z$.
\end{proof}

\begin{lemma}\label{uniform}
Suppose that condition (S7) of Theorem \ref{t1} holds and let $f\in\mathcal{C}_0$ be differentiable with $f'\in\mathcal{C}_0$ and recall
\begin{equation}
\ben t f(x)=t\int_{(0,1]} (f(-t\log(y+e^{-x/t}))-f(x))\,\nu(dy).
\end{equation}
and
\begin{equation}
\gen f(x)=\gamma\int_0^x(f(y)-f(x))1_{[0,\infty)}(x)\ .
\end{equation}
Then, for each $z>0$,
\begin{equation}
\lim_{t\downarrow 0}\sup_{x\in(-\infty,z]}\left|\ben tf(x)-\gen f(x)\right|=0\ .
\end{equation}
\end{lemma}

\begin{proof}
Denote $\|f'\|\equiv\sup_{x\in\mathbb{R}}|f'(x)|$ ($<\infty$ as $f'\in\mathcal{C}_0$).
Since $|f(x)-f(y)|\le \|f'\||x-y|$ then, for $0<y\le 1$,
\begin{align}
|f(-t\log(y+e^{-x/t}))-f(x)|&\le \|f'\||-t\log(y+e^{-x/t}))-x|\nonumber\\
&=\|f'\||\log(y+e^{-x/t}))+\log e^{x/t}|t\nonumber\\
&=\|f'\|t\log(ye^{x/t}+1)\\ &\le \|f'\|tye^{x/t}\ .\nonumber
\end{align}
Thus, recalling that
\begin{equation}\rho\equiv\int_{(0,1]}y\nu(dy)\left(=\int_0^1\nu(y,1]dy\right)<\infty\ ,
\end{equation}
we have that for $x\le 0$
\begin{equation}
|\ben t f(x)|\le \|f'\|\rho t^2e^{x/t}\le \|f'\|\rho t^2
\end{equation}
Since $\gen f(x)=0$ for $x\le 0$, this implies that
\begin{equation}\label{eq:neg1}\lim_{t\downarrow 0}\sup_{x\le 0}|\ben t f(x)-\gen  f(x)|= 0\end{equation}
as $t\to 0$.

Next, note that for $x\ge 0$,
\begin{align}
&\int_{(0,1]}(f(-t\log(y+e^{-x/t}))-f(x))\,\nu(dy)\nonumber\\ &=-\int_{(0,1]}\int^x_{-t\log(y+e^{-x/t})}f'(u)du\nu(dy) \\
&=-\int_{-t\log(1+e^{-x/t})}^xf'(u)\nu\left(e^{-u/t}-e^{-x/t},1\right]du\nonumber
\end{align}

In particular, upon substituting $y=e^{-u/t}-e^{-x/t}$, so that \[dy=-e^{-u/t}du/t=-(y+e^{-x/t})du/t\ ,\]
we have that
\begin{align}\label{eq:neg2}
&\left|\int_{-t\log(1+e^{-x/t})}^0f'(u)\nu(e^{-u/t}-e^{-x/t},1]du\right|\nonumber\\ &\le \|f'\|\int_{-t\log(1+e^{-x/t})}^0\nu(e^{-u/t}-e^{-x/t},1]du\nonumber\\
&=t\|f'\|\int_{1-e^{-x/t}}^1\frac{\nu(y,1]}{y+e^{-x/t}}dy\\
&\le t\|f'\|\int_{1-e^{-x/t}}^1\nu(y,1]dy\nonumber\\
&\le t\|f'\|\int_0^1\nu(y,1]dy=t\|f'\|\rho\ .\nonumber
\end{align}
The last expression clearly vanishes as $t\downarrow 0$ and in particular when multiplying it by $t$. Thus the left side converges to zero uniformly on $x\in[0,\infty)$.

From (\ref{eq:neg1}), (\ref{eq:neg2}) and
\begin{equation}
\gen f(x)=\gamma\int_0^x(f(y)-f(x))dx=-\gamma\int_0^xf'(u)udu\ ,
\end{equation}
it remains to show that for each $z>0$
\begin{equation}
\lim_{t\downarrow0}\sup_{x\in[0,z]}\left|\int_0^xf'(u)\left(\gamma u-t\nu\left(e^{-u/t}-e^{-x/t},1\right]\right)du\right|=0\ .
\end{equation}
We clearly have that
\begin{align}
&\left|\int_0^xf'(u)\left(\gamma u-t\nu\left(e^{-u/t}-e^{-x/t},1\right]\right)du\right|\nonumber\\ &\le \|f'\|\int_0^x\left|\gamma u-t\nu\left(e^{-u/t}-e^{-x/t},1\right]\right|du\ .
\end{align}
Substituting $y=e^{-u/t}-e^{-x/t}$, adding and subtracting $\gamma\log y$ in the second line of the following equation and rearranging terms give
\begin{align}\label{eq:terms}
&\int_0^x\left|\gamma u-t\nu\left(e^{-u/t}-e^{-x/t},1\right]\right|du\nonumber\\
&=t^2\int_0^{1-e^{-x/t}}\left|-\gamma \log(y+e^{-x/t})-\nu(y,1]\right|\frac{dy}{y+e^{-x/t}}\nonumber\\
&\le \gamma t^2\int_0^1\left|1-\frac{\nu(y,1]}{-\gamma\log y}\right|\frac{-\log y}{y+e^{-x/t}}\,dy\\
&\qquad+\gamma t^2\int_0^1\frac{\log(y+e^{-x/t})-\log y}{y+e^{-x/t}}dy \ .\nonumber
\end{align}
Substituting $y=e^{-x/t}v$ gives
\begin{align}\label{eq:log}
\int_0^1\frac{\log(y+e^{-x/t})-\log y}{y+e^{-x/t}}dy&=\int_0^{e^{x/t}}\frac{\log(v+1)-\log v}{v+1}dv\nonumber\\
&\le\int_0^\infty\frac{\log(v+1)-\log v}{v+1}dv\ .
\end{align}
Since $\int_0^\epsilon (-\log v)dv=\epsilon(1-\log\epsilon)<\infty$ and since
\begin{equation}
\frac{\log(v+1)-\log v)}{v+1}=\frac{1}{v+1}\int_v^{v+1}\frac{1}{u}du\le \frac{1}{v^2}
\end{equation}
it follows that the right hand side of (\ref{eq:log}) is finite and thus the second term of the right hand side of (\ref{eq:terms}) converges to zero uniformly on $x\in[0,\infty)$.
Therefore, as the first term on the right hand side of (\ref{eq:terms}) is bounded above (on $x\in [0,z]$) by
\begin{align}
\gamma t^2\int_0^1\left|1-\frac{\nu(y,1]}{-\gamma\log y}\right|\frac{-\log y}{y+e^{-z/t}}\,dy\label{wowo}
\end{align}
it remains to show that \eqref{wowo} vanishes as $t\downarrow0$.

Clearly, for any $\delta\in(0,1)$,
\begin{align}\label{eq:delta1}
\int_\delta^1\left|1-\frac{\nu(y,1]}{-\gamma\log y}\right|\frac{-\log y }{y+e^{-z/t}}\,dy\le\int_\delta^1\left|1-\frac{\nu(y,1]}{-\gamma\log y}\right|\frac{-\log y}{y}\,dy<\infty\ ,
\end{align}
so that upon multiplying by $t^2$ the left side converges to zero. Also, note that
\begin{align}
\int_0^{e^{-z/t}}\frac{-\log y}{y+e^{-z/t}}\,dy&\le e^{z/t}\int_0^{e^{-z/t}}(-\log y)dy\nonumber\\ &=e^{z/t}\cdot e^{-z/t}\left(1-\log e^{-z/t}\right)=1+\frac{z}{t}
\end{align}
which, upon multiplication by $t^2$, vanishes as $t\downarrow 0$. Therefore, also
\begin{align}\label{eq:e-z}
\gamma t^2\int_0^{e^{-z/t}}\left|1-\frac{\nu(y,1]}{-\gamma\log y}\right|\frac{-\log y}{y+e^{-z/t}}\,dy
\end{align}
vanishes as $t\downarrow0$, since by the assumptions $\left|1-\frac{\nu(y,1]}{-\gamma\log y}\right|$ is bounded on $[0,z]$.

To complete the proof, in view of (\ref{eq:delta1}) and (\ref{eq:e-z}), it remains to show that for any $\epsilon>0$ there is some $\delta>0$ and some $T>0$, such that for all $0<t<T$
\begin{align}
t^2\int_{e^{-z/t}}^\delta\left|1-\frac{\nu(y,1]}{-\gamma\log y}\right|\frac{-\log y}{y+e^{-z/t}}\,dy<\epsilon\ .
\end{align}
By the assumption we can pick some $0<\delta<1$ such that, for all $0<y<\delta$,
\begin{align}
\left|1-\frac{\nu(y,1]}{-\gamma\log y}\right|<\frac{\epsilon}{z^2}\ .
\end{align}
Then, take $T=\frac{z}{-\log\delta}$ and note that $t<T$ if and only if $e^{-z/t}<\delta$. We now have that for all $0<t<T$,
\begin{align*}
t^2\int_{e^{-z/t}}^\delta\left|1-\frac{\nu(y,1]}{-\gamma\log y}\right|\frac{-\log y}{y}dy&< \frac{\epsilon t^2}{z^2}\int_{e^{-z/t}}^\delta\frac{-\log y}{y}dy\\
&\le \frac{\epsilon t^2}{z^2} (-\log e^{-z/t})\int_{e^{-z/t}}^\delta\frac{1}{y}dy\\
&=\frac{\epsilon t}{z}\left(\log\delta-\log e^{-z/t}\right)\\
&=\epsilon (t\log\delta+1)< \epsilon\
\end{align*}
and the proof is complete.
\end{proof}

\noindent
{\bf Acknowledgement:} We thank Shaul Bar Lev and Thomas Kurtz for useful discussions.

\bibliographystyle{amsplain}
\bibliography{subext}
\end{document}